%% file: n-generic-arxiv-2.tex
\documentclass[twoside,a4paper]{amsart}
\usepackage{amssymb}
\usepackage{amsmath}
\usepackage[all,color]{xy}

\usepackage{ifpdf}
\ifpdf
  \usepackage[colorlinks,linkcolor=red,anchorcolor=blue,citecolor=green]{hyperref}
\else
\fi

\ifx\isdraft\undefined
\else
    \usepackage{fancyhdr}
\fi

\newtheorem{theorem}{Theorem}[section]

\newtheorem{claim}[theorem]{Claim}

\newtheorem{corollary}[theorem]{Corollary}
\newtheorem{lemma}[theorem]{Lemma}

\newtheorem{question}[theorem]{Question}

\theoremstyle{definition}

\numberwithin{equation}{section}

\newcommand{\<}{\langle}
\renewcommand{\>}{\rangle}
\newcommand{\uh}{\upharpoonright}

\title{Relative Definability of $n$-Generics}

\author{Wei Wang}

\thanks{This research was partially supported by NSF Grant 11471342 of China. The author thanks Andr\'{e} Nies for his comments on a previous version of this paper.}

\address{Institute of Logic and Cognition and Department of Philosophy, Sun Yat-sen University, 135 Xingang Xi Road, Guangzhou 510275, P.R. China}
\email{wwang.cn@gmail.com}


\begin{document}

\begin{abstract}
A set $G \subseteq \omega$ is $n$-generic for a positive integer $n$ if and only if every $\Sigma^0_n$ formula of $G$ is decided by a finite initial segment of $G$ in the sense of Cohen forcing. It is shown here that every $n$-generic set $G$ is properly $\Sigma^0_n$ in some $G$-recursive $X$. As a corollary, we also prove that for every $n > 1$ and every $n$-generic set $G$ there exists a $G$-recursive $X$ which is generalized ${\rm low}_n$ but not generalized ${\rm low}_{n-1}$. Thus we confirm two conjectures of Jockusch \cite{Jockusch:1980.generic}.
\end{abstract}

\maketitle

\section{Introduction}

In recursion theory, Cohen generics have been intensively studied and are usually just called \emph{generics}. In \cite{Jockusch:1980.generic} Jockusch presented the first systematic study of the Turing degrees of generics and introduced the notion of \emph{$n$-genericity}. A set $G$ is $n$-generic for a positive $n$ if and only if every $\Sigma^0_n$ or $\Pi^0_n$ fact of $G$ is forced by a finite initial segment of $G$. For recursion theorists, it is natural to investigate which properties of generics hold for $n$-generics and which techniques applied to generics could also be applied to $n$-generics. On the other hand, $n$-generics are more handy for most recursion theoretic techniques. So $n$-generics, in particular $1$-generics, are even more intensively studied and play important roles in recursion theory.

An interesting recursion theoretic property of $1$-generics is their relative recursive enumerability. Recall that $X$ is \emph{recursively enumerable in and above} (r.e.a.\ for short) $Y$ if $X$ is $Y$-r.e.\ and $Y <_T X$. In terms of definability, if $X$ is r.e.a.\ $Y$ then $X$ is properly $\Sigma^0_1$ in $Y$. An \emph{r.e.a.}\ set is one which is r.e.a.\ some other set, and an \emph{r.e.a.\ degree} is a Turing degree containing an r.e.a.\ set. In the theory of Turing degrees, the r.e.a.\ degrees have been proven important. Jockusch \cite{Jockusch:1980.generic} proved that every $1$-generic set $G$ is r.e.a.\ and thus the $1$-generic degrees form a subclass of the r.e.a.\ degrees. In the same paper, it was also shown that every $2$-generic $G$ is properly $\Sigma^0_2$ in some $G$-recursive $X$. The first result connects $1$-generic degrees with recursively enumerable degrees. For example, combining the relative recursive enumerability of $1$-generics and a theorem of r.e.\ degrees (Yates \cite{Yates:70*1}) Jockusch concluded that the Turing degrees below a $1$-generic degree are not dense (\cite[Corollary 5.3]{Jockusch:1980.generic}). Moreover, as the r.e.a.\ degrees form a large and well-known class (e.g., see \cite{Cai.Shore:2012}), the first result seems more popular. However, the proofs of the two results are similar. Both proofs were done by cleverly defining some functionals which produce the desired $X$ for sufficient generic $G$. But these clever functionals do not seem to admit easy generalizations to large $n$. So Jockusch posted a general conjecture that every $n$-generic $G$ is properly $\Sigma^0_n$ in some $G$-recursive $X$ for every positive integer $n$ (\cite[Conjecture 5.13]{Jockusch:1980.generic}).

The main purpose of this article is to confirm the above conjecture of Jockusch. We provide a proof in \S \ref{s:Main-theorem}. Then we present some consequences of our main theorem in \S \ref{s:Remarks}, including a confirmation of \cite[Conjecture 5.14]{Jockusch:1980.generic}. We also post some open questions in \S \ref{s:Remarks}.

Our terminology is rather standard. A set $X \subseteq \omega$ is usually identified with its characteristic function in $2^\omega$. We write $\sigma \subseteq \tau$ if $\sigma$ is an initial segment of $\tau \in 2^{<\omega}$ and $\sigma \subset \tau$ if $\sigma \subseteq \tau$ and $\sigma \neq \tau$. The concatenation of two finite sequences $s$ and $t$ is simply denoted by $s t$. If $x_0, \ldots, x_{n-1} \in \omega$ then $\<x_0 \ldots x_{n-1}\>$ denote the finite sequence of length $n$ with each $x_i$ being its $i$-th element. A set $D \subseteq 2^{<\omega}$ is \emph{dense} if every $\sigma \in 2^{<\omega}$ is extended by some $\tau \in D$. For a fixed $\rho \in 2^{<\omega}$, a set $D \subseteq 2^{<\omega}$ is \emph{dense below $\rho$} if every extension of $\rho$ is extended by some element of $D$. For a fixed $G \subseteq \omega$, a set $D \subseteq 2^{<\omega}$ is \emph{dense along $G$} if every finite initial segment of $G$ is extended by some element of $D$. Throughout the paper, forcing always means Cohen forcing.

Besides \cite{Jockusch:1980.generic}, we also recommend Kumabe \cite{Kumabe:1996} for background of $n$-generic degrees.

\section{The Main Theorem}\label{s:Main-theorem}

In this section we prove the following theorem.

\begin{theorem}\label{thm:Main}
For each $n > 0$, there exists a partial recursive functional $\Gamma$ such that if $G$ is $n$-generic then $\Gamma(G)$ is total and $G$ is properly $\Sigma^0_n$ in $\Gamma(G)$.
\end{theorem}

To prove Theorem \ref{thm:Main}, we recursively construct two functionals $\Gamma$ and $\Delta$ such that the followings hold for every $n$-generic $G$:
\begin{enumerate}
    \item[(A1)] $X = \Gamma(G)$ is total.
    \item[(A2)] $\omega - G$ is not $\Sigma^X_n$.
    \item[(A3)] $\Delta_k(X; x_0,\ldots,x_{n-k})$ is total for each $k < n-1$, where
    \begin{gather*}
	\Delta_0(X; x_0,\ldots,x_{n-1}) = \Delta(X; x_0,\ldots,x_{n-1}),\\
	\Delta_{k+1}(X; x_0,\ldots,x_{n-k-2}) = \lim_{x_{n-k-1}} \Delta_{k}(X; x_0,\ldots,x_{n-k-2},x_{n-k-1}).
    \end{gather*}
    \item[(A4)] $x \in G \leftrightarrow \Delta_{n-1} (X; x) = 1$ for every $x$.
\end{enumerate}
By (A3-4), $G$ is $\Sigma^X_n$.

In the construction, we refine (A1-4) to countably many requirements and recursively assign to each requirement countably many stages so that every stage is assigned to exactly one requirement. A stage assigned to a requirement $R$ is called an \emph{$R$-stage}. At each stage we construct finite approximations of $\Gamma$ and $\Delta$. If stage $s$ is an $R$-stage then \emph{the $R$-module} performs a finite set of actions for the sake of meeting requirement $R$.

Firstly, we sketch the basic strategies for (A1-4). For convenience, we write $\vec{x}$ for a tuple $(x_0,\ldots,x_{k-1})$ where $k$ is the length of $\vec{x}$.

\subsection{Meeting (A1)}

We view $\Gamma$ as a partial recursive function $\Gamma: 2^{<\omega} \to 2^{<\omega}$ which is continuous, i.e., $\Gamma(\sigma) \subseteq \Gamma(\tau)$ whenever $\sigma \subseteq \tau$. To meet (A1), we require that
\begin{equation}\label{eq:A1}
    \{\sigma: |\Gamma(\sigma)| > l\} \text{ is dense in } 2^{<\omega} \text{ for each } l.
\end{equation}
So we have the following requirements:
\begin{equation*}
    P(\sigma,l): (\exists \tau \supset \sigma) |\Gamma(\tau)| > l.
\end{equation*}
A $P(\sigma,l)$-module picks some $\tau \supset \sigma$ and defines $\Gamma(\tau)$ as some $\xi$ of length at least $l+1$. The set in \eqref{eq:A1} will be not only dense but also recursive. Hence $\Gamma(G)$ is total for every $1$-generic $G$.

\subsection{Meeting (A2)}

Suppose that $X = \Gamma(G)$ is total. We replace (A2) with the following requirement for each $\Sigma^0_n$ formula $\theta$:
\begin{multline}\label{eq:A2}
    (\forall x \not\in G)(\exists \sigma \supset G \uh x) \sigma \Vdash \theta(\Gamma(G), x) \to \\
    \{\sigma: (\exists x) \sigma(x) = 1 \wedge \sigma \Vdash \theta(\Gamma(G), x)\} \text{ is dense along } G.
\end{multline}
As $n > 0$ and $\Gamma$ satisfies \eqref{eq:A1}, $\sigma \Vdash \theta(\Gamma(G), x)$ is a $\Sigma^0_n$ predicate of $G, \sigma$ and $x$. Assume that $G = \{x: \neg \ \theta(\Gamma(G), x)\}$ for a $\Sigma^0_n$ formula $\theta$. By the $n$-genericity of $G$, the premise of \eqref{eq:A2} holds. Applying \eqref{eq:A2} and the $n$-genericity of $G$ again, there exists $\sigma \subset G$ such that $\sigma(x) = 1$ and $\sigma \Vdash \theta(\Gamma(G),x)$ and thus $\theta(\Gamma(G),x)$ holds. So we have a desired contradiction. This proves that \eqref{eq:A2} implies (A2).

The strategy to meet \eqref{eq:A2} is as follows: For each $\rho$, we construct a recursive set $D \subset 2^{<\omega}$ dense below $\rho\<0\>$, and for each $\sigma \in D$ assign some $\tau \supseteq \rho \<1\>$ and guarantee that if $\sigma$ forces a $\Pi^0_{n-1}$ formula $\varphi$ of $\Gamma(G)$ then $\tau$ forces $\varphi$ too. We call this $\tau$ the \emph{deputy} of $\sigma$. Assume that every deputy meets our expectation and the premise of \eqref{eq:A2} holds. We write each $\Sigma^0_n$ formula $\theta$ as $(\exists y) \varphi$ where $\varphi$ is $\Pi^0_{n-1}$. Fix $x \not\in G$ and let $\rho = G \uh x$, and assume that $\sigma \Vdash \theta(\Gamma(G), x)$ for some $\sigma \supset \rho$. Then either $\sigma(x) = 1$ or some $\sigma' \supseteq \sigma$ forces $\varphi(\Gamma(G),x,y)$ and has a deputy $\tau \supset \rho\<1\>$. Hence $\tau(x) = 1$ and $\tau \Vdash \varphi(\Gamma(G),x,y)$ as well. So the above strategy guarantees \eqref{eq:A2}.

Suppose that $\rho \subset \rho \<0\> \subset \sigma$ and $\tau$ is the deputy of $\sigma$. In order to figure out how to make $\tau$ do its job as the deputy of $\sigma$, let us begin with small $n$'s.
\begin{itemize}
    \item If $n = 1$ and $\varphi$ is a $\Pi^0_0$ formula then there is a recursive subset $P$ of $2^{<\omega}$ such that $\sigma \Vdash \varphi(\Gamma(G))$ if and only if $\Gamma(\sigma)$ extends an element of $P$. So it suffices to define $\Gamma(\tau) \supset \Gamma(\sigma)$. Note that this is the key ingredient of Lemma 5.2 in Jockusch \cite{Jockusch:1980.generic}.
    \item Suppose that $n = 2$. We need another layer of deputies. To distinguish deputies at different layers, we call $\tau$ the $1$-deputy of $\sigma$. We pick $0$-deputies $\mu \supset \sigma$ for $\nu$ densely below $\tau$ and make these deputies act as in the above paragraph. To ensure that the above paragraph could be applied to $0$-deputies, we certainly need that $\Gamma(\tau) \supseteq \Gamma(\sigma)$. With this strategy, if $\varphi = (\forall y) \psi$ where $\psi$ is $\Sigma^0_0$ and $\tau \not\Vdash \varphi(\Gamma(G))$, then some $\nu \supset \tau$ forces some $\neg \psi(\Gamma(G),k)$ and so does the $0$-deputy $\mu \supset \sigma$ of some extension of $\nu$. Hence $\sigma \not\Vdash \varphi(\Gamma(G))$ either. The reader may have noticed the similarity between the above argument and that following the proof of \cite[Lemma 5.10]{Jockusch:1980.generic}.
\end{itemize}
Note that each $\sigma$ may extend more than one $\rho \<0\>$. So to be precise, the corresponding $(n-1)$-deputy should be called something like the $(\rho,n-1)$-deputy of $\sigma$. For simplicity, we only define deputies for sequences of fresh lengths at each stage. We also require that each $\sigma$ has at most one deputy and thus it cannot have deputies at different layers. So the deputy of $\sigma$ is unambiguous. But we may still use the terms $k$-deputy and $(\rho,k)$-deputy occasionally. For bookkeeping purpose, we simultaneously construct a partial recursive injection $d$ such that if $d(\sigma)$ is defined then it is the deputy of $\sigma$.

Based on these simple examples, for fixed positive integer $n$ and every $\rho$ and $x = |\rho|$, we define a $(\rho,n-1)$-deputy extending $\rho\<1\>$ for each $\sigma$ in a recursive set dense below $\rho\<0\>$. If $0 < k < n$ and $\tau$ is defined as the $k$-deputy of $\sigma$ then for $\nu$ in a recursive set dense below $\tau$ we define its $(\tau,k-1)$-deputy $\mu = d(\nu)$ to be some extension of $\sigma$. As in the above examples, if $\mu = d(\nu)$ then we always ensure that $\Gamma(\mu) \supset \Gamma(\nu)$.

We present the above strategy in a more formal fashion as follows. Firstly we design a hierarchy of requirements:
\begin{gather*}
    D(0, \sigma, \tau): \Gamma(\sigma) \subseteq \Gamma(\tau),\\
    D(k, \sigma, \tau): \{\nu: (\exists \mu \supset \sigma)(D(k-1, \nu, \mu) \text{ is satisfied})\} \text{ is dense below } \tau
\end{gather*}
for $0 < k < n$, and
\begin{equation*}
    D(\rho): \{\sigma: (\exists \tau \supset \rho \<1\>)(D(n-1,\sigma,\tau) \text{ is satisfied})\} \text{ is dense below } \rho \<0\>.
\end{equation*}
A $D(k,\sigma,\tau)$ applies only if $\tau$ is defined as the $k$-deputy of $\sigma$ and requires that $\tau$ fulfills its job as a deputy. The $D(\rho)$ requires that $(n-1)$-deputies are defined for elements in a set dense below $\rho\<0\>$.

A $D(\rho)$-module acts as below:
\begin{enumerate}
    \item Pick the least $\mu \supset \rho \<0\>$ such that no $\sigma \supseteq \mu$ has a $(\rho,n-1)$-deputy defined.
    \item Pick $\sigma \supset \mu$ with fresh length and define $d(\sigma)$, the $(n-1)$-deputy of $\sigma$, as some $\tau \supset \rho\<1\>$.
    \item Extend the approximation of $\Gamma$ so that $\Gamma(\sigma) \subset \Gamma(\tau)$.
\end{enumerate}

For $D(k,\sigma,\tau)$ with $0 < k < n$, if at a $D(k,\sigma,\tau)$-stage $\tau$ is already defined as the $k$-deputy of $\sigma$ then the $D(k,\sigma,\tau)$-module acts like a $D(\rho)$-module:
\begin{enumerate}
    \item Pick the least $\bar{\nu} \supset \tau$ such that no $\nu \supseteq \bar{\nu}$ has its $(\tau,k-1)$-deputy defined.
    \item Pick $\nu \supset \bar{\nu}$ with fresh length and define $d(\nu)$, the $(k-1)$-deputy of $\nu$, as some $\mu \supset \sigma$.
    \item Extend the approximation of $\Gamma$ so that $\Gamma(\mu) \supset \Gamma(\nu)$.
\end{enumerate}

A $D(0,\sigma,\tau)$-module simply does nothing, because if $\tau = d(\sigma)$ then $D(0,\sigma,\tau)$ is automatically satisfied when $d(\sigma)$ is defined by a $D(\rho)$- or $D(1,-)$-module.

Figure \ref{fig:deputies} shows a partial sample of the construction of deputies for $n = 2$, where $\tau_i$'s are the $1$-deputies of $\sigma_i$'s and $\mu_i$'s the $0$-deputies of $\nu_i$'s. It might worth notice that in the sample given in Figure \ref{fig:deputies} the $0$-deputies of $\nu_i$'s extending $\tau_0$ need not to form a set dense below $\sigma_0$ though the $\nu_i$'s having deputies are required to be dense below $\tau_0$. In general, if $\tau$ is the $k$-deputy of $\sigma$ and $k > 0$, then the extensions of $\tau$ having $(k-1)$-deputies are necessarily dense below $\tau$, but their deputies are not dense below $\sigma$.

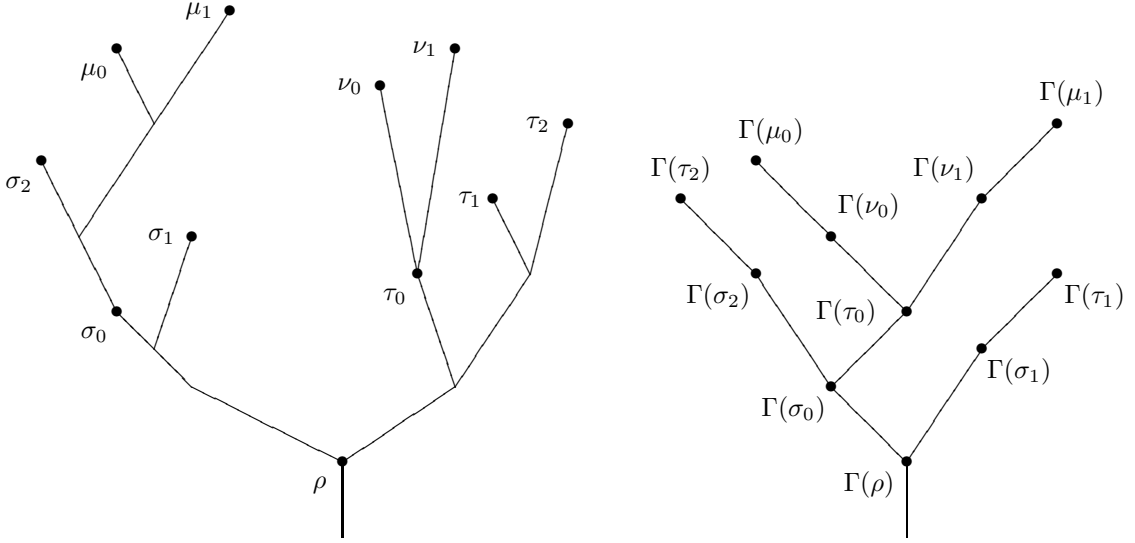
\begin{figure}
    \input{diagram-deputies.tex}
    \caption{The construction of deputies for $n = 2$.}
    \label{fig:deputies}
\end{figure}

\begin{lemma}\label{lem:good-deputy}
If $D(n-k,\sigma,\tau)$ is satisfied and $\sigma \Vdash \varphi(\Gamma(G))$ for some $\Pi^0_{n-k}$ formula $\varphi$ then $\tau \Vdash \varphi(\Gamma(G))$.
\end{lemma}

\begin{proof}
The case for $n-k = 0$ holds as in the above examples.

Suppose that $n-k > 0$, $D(n-k,\sigma,\tau)$ is satisfied and $\tau \not\Vdash \varphi(\Gamma(G))$ for some $\Pi^0_{n-k}$ formula $\varphi$. Write $\varphi$ as $(\forall y) \psi$ where $\psi$ is $\Sigma^0_{n-k-1}$. Then there exist $j$ and $\bar{\nu} \supset \tau$ such that $\bar{\nu} \Vdash \neg \psi(\Gamma(G),j)$. Pick $\nu \supseteq \bar{\nu}$ and $\mu \supset \sigma$ such that $D(n-k-1,\nu,\mu)$ is satisfied. By the induction hypothesis, $\mu \Vdash \neg \psi(\Gamma(G),j)$ too. So $\sigma \not\Vdash \varphi(\Gamma(G))$.
\end{proof}

\subsection{Meeting (A3-4)}\label{ss:Q}

To meet (A3-4), we have requirements demanding the convergence of $\Delta_k(\Gamma(G); \vec{x})$ for appropriate $\vec{x}$'s. Each module of these requirements is assigned a specific triple $(k,\sigma,\vec{x})$ and picks a pair $(\tau,i)$ such that $\tau \supset \sigma$ and $i < 2$, then it tries to guarantee that $\tau \Vdash \Delta_k(\Gamma(G); \vec{x}) \downarrow = i$. To achieve their goals, these modules must collaborate. Suppose that $k > 0$ and $\tau \Vdash \Delta_k(\Gamma(G); \vec{x}) \downarrow = i$ is to be ensured. The responsible module $M$ (of $k$-th layer) picks a threshold $\bar{x}$ and demands that if a $(k-1)$-th layer module $M'$ works to guarantee $\xi \Vdash \Delta_{k-1}(\Gamma(G); \vec{x},x) \downarrow = j$ for $\xi$ extending $\tau$ and $x > \bar{x}$ then $j$ must equal $i$. Such a demand by the $k$-th layer module $M$ is called a $(k-1)$-th layer \emph{constraint}. For $k=1$ the constraint means that if $\Delta(\Gamma(\xi); \vec{x},x)$ is defined for $\xi \supset \tau$ and $x > \bar{x}$ then its value must be $i$. For $k > 1$, the module $M'$ imposes a $(k-2)$-th layer constraint similarly. So the collaboration is carried out by imposing cascading constraints and obeying constraints by modules of higher layers.

Below we refine the above strategy. We shall have $n$-layers of requirements. But by (A2), $\Delta_{n-1}$ cannot be total. So the $(n-1)$-th layer is different. At the $(n-1)$-th layer, for each $\sigma$ and $x$ such that $|\sigma| > x$, we have the following requirement
$$
  Q(n-1,\sigma,x): \sigma(x)=1 \leftrightarrow (\exists \tau \supseteq \sigma) \tau \Vdash \Delta_{n-1}(\Gamma(G); x) \downarrow = 1.
$$
The $Q(n-1,\sigma,x)$-module acts according to the first of the followings that applies:
\begin{itemize}
    \item If $n = 1$, $\sigma(x) = 1$ and $\Delta(\Gamma(\tau); x) \uparrow$ for all $\tau \supseteq \sigma$ then pick $\tau \supset \sigma$ and extend $\Gamma$ and $\Delta$ so that $\Delta(\Gamma(\tau); x) = 1$. If $n = 1$ and $\sigma(x) = 0$ then do nothing.
    \item Assume that $n > 1$ and $\sigma(x) = 1$. Pick $\tau \supset \sigma$ of fresh length and a fresh \emph{threshold} $\bar{x}$ and impose a \emph{$Q$-constraint} at the $(n-2)$-th layer denoted by $c(n-2,\tau,x,\bar{x},1)$, unless there is already some $c(n-2,\zeta,x,-)$ such that $\zeta$ and $\sigma$ are comparable (in this case either $\sigma$ or an extension of $\sigma$ is expected to force $\Delta_{n-1}(\Gamma(G); x) \downarrow$). All $Q(n-2,\rho,-)$-modules with $\rho$ extending $\tau$ will obey the above constraint. As we shall see, obeying $c(n-2,\tau,x,\bar{x},1)$ guarantees the desired forcing relation, i.e., $\tau \Vdash \Delta_{n-1}(\Gamma(G); x) \downarrow = 1$ or more precisely $\tau \Vdash (\forall x_1 > \bar{x}) \Delta_{n-2}(\Gamma(G); x, x_1) \downarrow = 1$. In the next paragraph we shall see how a $Q$-module at lower layer obeys constraints imposed by $Q$-modules at higher layers.
    \item Suppose that $n > 1$ and $\sigma(x) = 0$. Pick $\tau$ and $\bar{x}$ as above and also a fresh $x_1$. If $n = 2$ then extend $\Delta$ and $\Gamma$ so that $\Delta(\Gamma(\tau); x,x_1) \downarrow = 0$, and if $n > 2$ then impose a constraint $c(n-3,\tau,\<x x_1\>,\bar{x},0)$, so that $\tau \Vdash \Delta_{n-2}(\Gamma(G); x,x_1) \downarrow = 0$. In this way, the $Q(n-1,\rho,x)$-modules for $\rho \supset \sigma$ together prevent $\sigma$ forcing $\Delta_{n-1}(\Gamma(G); x) \downarrow = 1$.
\end{itemize}

If $0 < k < n-1$, for every $\sigma$ and every $\vec{x}$ such that $|\vec{x}| = n-k$ and $\max \vec{x} < |\sigma|$, we have the following requirement:
$$
    Q(k,\sigma,\vec{x}): (\exists \tau \supset \sigma, i < 2) \tau \Vdash \Delta_{k}(\Gamma(G); \vec{x}) \downarrow = i.
$$
The $Q(k,\sigma,\vec{x})$-module firstly picks $\tau \supset \sigma$ of fresh length and a fresh threshold $\bar{x}$. Then it acts according to the first of the following three cases that applies:
\begin{itemize}
    \item If there exists $c(k-1,\xi,\vec{x},-)$ such that $\xi$ and $\sigma$ are comparable then it does nothing, since for $\tau$ being the longest of $\sigma$ and $\xi$ we can expect $\tau$ witnessing the satisfaction of the requirement.
    \item If some $c = c(k,\xi,\vec{x} \uh (n-k-1),\bar{y},i)$ is active at $\sigma$ and $\bar{y} < x_{n-k-1}$, $x_{n-k-1}$ being the last element of $\vec{x}$, then this $Q$-module obeys the constraint $c$ by imposing a new constraint $c(k-1,\tau,\vec{x},\bar{x},i)$. In other words, if some $Q$-module at layer $k+1$ demands that $\xi \Vdash (\forall x > \bar{y}) \Delta_k(\Gamma(G); x_0,\ldots,x_{n-k-2},x) \downarrow = i$ then the $Q(k,\sigma,\vec{x})$-module demands that $\tau \Vdash \Delta_k(\Gamma(G); \vec{x}) \downarrow = i$.
    \item If none of the above cases applies then the $Q(k,\sigma,\vec{x})$-modules imposes a constraint $c(k-1,\tau,\vec{x},\bar{x},\sigma(x_0))$.
\end{itemize}

The $Q(0,\sigma,x)$'s for $n=1$ are instances of $Q(n-1,-)$'s. Suppose that $n > 1$. For each $\sigma$ and $\vec{x}$ such that $|\vec{x}| = n$ and $|\sigma| > \max \vec{x}$, we have the following requirement:
\begin{equation*}
    Q(0,\sigma,\vec{x}): (\exists \tau \supset \sigma, i < 2) \Delta(\Gamma(\tau); \vec{x}) \downarrow = i.
\end{equation*}
A $Q(0,\sigma,\vec{x})$-module acts only if $\Delta(\Gamma(\tau); \vec{x})$ is undefined for any $\tau$ extending $\sigma$. If it acts, the $Q(0,\sigma,\vec{x})$-module firstly picks $\tau \supset \sigma$ of fresh length. If a constraint $c(0,\xi,\vec{x} \uh n-1,\bar{y},i)$ is active at $\sigma$ and $\bar{y} < x_{n-1}$, $x_{n-1}$ being the last element of $\vec{x}$, then this $Q$-module defines $\Delta(\Gamma(\tau); \vec{x}) \downarrow = i$ so that it obeys the constraint above; otherwise it defines $\Delta(\Gamma(\tau); \vec{x}) \downarrow = \sigma(x_0)$, $x_0$ being the first element of $\vec{x}$.

If the construction were carried out only by $Q$-modules, $\Delta_k(\Gamma(G))$ would be total for all $k < n-1$ and all $1$-generic $G$ and the expected forcing relations would be secured. It seems that we would even have the totality of $\Delta_{n-1}(\Gamma(G))$ and
\begin{equation}\label{eq:Delta-n-1-total}
    G(x_0) = \Delta_{n-1}(\Gamma(G); x_0).
\end{equation}
However, \eqref{eq:Delta-n-1-total} contradicts (A2), since the totality of each $k < n - 1$ implies that $\Delta_{k}(\Gamma(G); x_0,\ldots,x_{n-k-1}) \downarrow = i$ is a $\Sigma^0_{k+1}$ predicate in $\Gamma(G)$ for each $k < n$. This contradiction will be dissolved in the next subsection when we try to bring the $D$- and $Q$-modules together.

\subsection{Conflicts}\label{ss:conflicts}

To integrate the modules, we must resolve the following issues.
\begin{enumerate}
    \item The functionals $\Gamma$ and $\Delta$ are defined by various modules. For the consistency of the two functionals, all modules should follow some code when defining $\Gamma$ and $\Delta$.
    \item From subsection \ref{ss:Q}, we can see that the $Q$-modules are designed to collaborate, i.e., lower layer modules obey constraints imposed by higher layer ones. However, the construction of deputies by $D$-modules could introduce some problems. For, if a $D(k,\sigma,\tau)$-module defines a deputy $\mu = d(\nu)$ for $\nu \supset \tau$, then $\Gamma(\mu)$ and $\Delta(\Gamma(\mu))$ extend $\Gamma(\nu)$ and $\Delta(\Gamma(\nu))$ respectively, and it is not obvious that $\Delta(\Gamma(\nu))$ is submitted to $Q$-constraints active at $\mu$.
    \item The problem discussed at the end of the previous subsection \ref{ss:Q}.
\end{enumerate}

For (1), we setup two codes to be followed by every module at every stage:
\begin{itemize}
    \item[(C1)] $\Gamma(\sigma) = \Gamma(\sigma\<1^k\>)$ for all $\sigma$ and $k$.
    \item[(C2)] $\Delta(\xi) = \Delta(\xi\<1^k\>)$ for all $\xi$ and $k$, i.e., $\Delta(\xi\<1^k\>; \vec{x}) \downarrow = i$ if and only if $\Delta(\xi; \vec{x}) \downarrow = i$.
\end{itemize}
If these codes are followed then it is easy to see that we can build valid functionals in some appropriate way. To follow (C1-2), we explicitly define $\Gamma(\sigma)$ and $\Delta(\sigma)$ only for $\sigma$ ending with $0$.

To examine issue (2) concretely, let us consider the guiding example below: $n = 3$, $\tau \supset \rho\<1\>$ is the $2$-deputy of $\sigma \supset \rho\<0\>$, and some $\nu$ extending $\tau$ has a $1$-deputy $\mu$ extending $\sigma$. Let $x = |\rho|$. Suppose that the $Q(2,\rho\<0\>,x)$-module imposes a constraint $c_0' = c(0,\eta',\<x x_1\>,\bar{x}_1',0)$ as $(\rho\<0\>)(x) = 0$. Since $\tau(x) = 1$, naturally the $Q(2,\tau,x)$-module may impose a constraint $c_1 = c(1,\zeta,x,\bar{x},1)$. Moreover, suppose that later a $Q(1,\upsilon,\<x x_1\>)$-module with $\upsilon \supset \zeta$ obeys $c_1$ and imposes another $c_0 = c(0,\eta,\<x x_1\>,\bar{x}_1,1)$. It could be that $c_0'$ is active at $\sigma$ and $c_0$ active at $\nu$, since $2$-deputies are defined for densely many extensions of $\rho\<0\>$ and $1$-deputies for densely many extensions of $\tau$. At some $\theta \supset \mu$, suppose that for $\vec{x} = \<x x_1 x_2\>$ with $x_2 > \max\{\bar{x}_1,\bar{x}_1'\}$, $c_0'$ is obeyed and $\Delta(\Gamma(\theta); \vec{x}) = 0$ is defined. Later the $D(1,\nu,\mu)$-module defines a $0$-deputy $d(\theta) \supset \nu$. Since $\Gamma(d(\theta)) \supset \Gamma(\theta)$, $\Delta(\Gamma(d(\theta)); \vec{x}) = 0$ too. So $c_0$ and the forcing relation $\eta \Vdash (\forall x_2 > \bar{x}_1) \Delta(\Gamma(G); x,x_1,x_2) = 1$ are broken though $c_0$ is active at $d(\theta)$.

But if we examine the motivation of deputies then we can see a natural solution. Recall that if $\tau$ is the $k$-deputy of $\sigma$ and $\sigma$ forces a $\Pi^0_k$ predicate $\varphi(\Gamma(G))$ of $\Gamma(G)$ then $\tau$ should force $\varphi(\Gamma(G))$ as well. In the above example, the $Q$-constraint $c_0'$ demands that $\eta'$ forces the $\Pi^0_2$ predicate
\begin{equation}\label{eq:Q-constraint-1}
    (\forall x_2 > \bar{x}_1') (\Delta(\Gamma(G); x,x_1,x_2) \downarrow = 0).
\end{equation}
Since the $Q$-modules make $\Delta(\Gamma(G))$ a total $2$-valued function, \eqref{eq:Q-constraint-1} is equivalent to the following $\Pi^0_1$ predicate of $\Gamma(G)$:
\begin{equation}\label{eq:Q-constraint-2}
    (\forall x_2 > \bar{x}_1') \neg (\Delta(\Gamma(G); x,x_1,x_2) \downarrow = 1).
\end{equation}
As $c_0'$ is active at $\sigma$ and $\tau$ is the $2$-deputy of $\sigma$, $\tau$ should also be demanded to force \eqref{eq:Q-constraint-2}. When the $D(\rho)$-module defines $\tau = d(\sigma)$, it should impose a $Q$-constraint $c(0,\tau,\<x x_1\>,\bar{x}_1',0)$, which has all its parameters except $\tau$ copied from $c_0'$ and thus is called a \emph{replica} of $c_0'$. This action is called \emph{replication} or \emph{replication of constraints}. For the consistency of this action, we need two new codes:
\begin{itemize}
    \item[(C3)] If a $D$-module defines a deputy $\tau = d(\sigma)$ then $\tau = \rho\<1^k0\>$ with a fresh $k$ and an appropriate $\rho$.
    \item[(C4)] If a $Q$-constraint is active at any $\rho\<1^k\>$ then it is active at $\rho$.
\end{itemize}
With (C3-4), when the $D(\rho)$-module in the guiding example defines $d(\sigma)$, it can define the deputy to be some $\tau = \rho\<1^l0\>$ such that every constraint active at $\rho\<1^l\>$ is active at $\rho$. Hence if the constraints active at $\sigma$ are consistent then after replication the constraints active at $\tau$ are consistent as well. In general, if a $D$-module defines a $k$-deputy $\tau = d(\sigma)$ and a $Q$-constraint $c = c(j,\zeta,-)$ is active at $\sigma$ with $j < k$, then this $D$-module should impose a new $Q$-constraint like $c(j, \tau, -)$ which is called a \emph{replica} of $c(j,\zeta,-)$ and has its parameters except $\tau$ identical with those of $c$. Actually, it suffices to replicate $c(j,\zeta,-)$ as above and satisfying an additional condition that $\rho \subset \zeta \subseteq \sigma$ if $\tau$ is defined as the $(\rho,k)$-deputy of $\sigma$.

Assume that the replication strategy is implemented and (C3-4) are followed. Let us examine the guiding example again. When the $D(\rho)$-module acts, it can pick $\tau$ so that no constraints like $c_1$ or $c_0'$ are active at $\tau$ and then it replicates constraints active at $\sigma$ to $\tau$. So if $c_0'$ were already active at $\sigma$ then a replica of $c_0'$ would be active at $\tau$ and thus no constraint like $c_0$ would ever appear above $\tau$. Hence the definition of $d(\theta)$ later would not break any constraint. Suppose that no constraint like $c_0'$ is active at $\sigma$ but $c_0$ is active at $\nu$. When the $D(2,\sigma,\tau)$-module acts, it defines $d(\nu)$ to be some $\mu \supset \sigma$ and replicates $c_0$ to $c_0'' = c(0,\mu,\<x x_1\>,\bar{x}_1,1)$. Hence if later $\Delta(\Gamma(\theta); \vec{x})$ is defined it must equal $1$. Thus $c_0$ will not be broken by the definition of $d(\theta)$. Thus we solve the conflict in the guiding example.

Notice that as in subsection \ref{ss:Q}, no $c(n-2,\sigma,x,\bar{x},0)$ would ever be imposed by $Q$-modules. So the replication strategy will never produce any $c(n-2,\tau,x,\bar{x},0)$. Furthermore, it is necessary that a $Q$-constraint $c(j,-)$ with $k \leq j$ active at $\nu$ will not be replicated to the $k$-deputy $\mu$ of $\nu$. As an illustration, in the example above, if $c_1$ were replicated to $\mu$ when $\mu$ is defined as the $1$-deputy of $\nu$ then we would have $\mu \Vdash \Delta_2(\Gamma(G); x) = 1$ while $\mu(x) = 0$.

Replication of constraints also solves the issue (3) above. Let us go back to the guiding example. Recall that $x = |\rho|$. By the $Q(2,-)$-modules, the following set is dense below $\rho\<0\>$:
$$
    \{\xi: \text{there exists some } c(0,\xi,\<x y_1\>,\bar{y}_1,0)\}.
$$
So we may assume that some $c(0,\xi,\<x y_1\>,\bar{y}_1,0)$ above is active at $\sigma$, meaning that $\sigma \Vdash \Delta_1(\Gamma(G); x,y_1) = 0$. By the $Q(1,-)$-modules, the following set is dense below $\zeta$:
$$
    \{\eta: \text{there exists some } c(0,\eta,\<x x_1\>,\bar{x}_1,1)\}.
$$
As $(\tau,1)$-deputies are defined for densely many extensions of $\zeta$, some $c(0,\eta,-)$ as above will have a replica active above $\sigma$, implying that some extension $\lambda$ of $\sigma$ will force $\Delta_1(\Gamma(G); x,x_1) = 1$ for some $x_1$. Hence, for a generic $G$ extending $\lambda$, we have $\Delta_1(\Gamma(G); x,y_1) = 0$ and $\Delta_1(\Gamma(G); x,x_1) = 1$. Due to the construction of deputies and replication of constraints, there are infinitely many such oscillations along $G$, causing $\Delta_2(\Gamma(G); x)$ undefined. On the other hand, no $\pi \supseteq \rho\<1\>$ could have a $2$-deputy above $\rho\<0\>$. So such oscillations cannot happen infinitely often to any generic $H \supset \rho\<1\>$ and $\Delta_2(\Gamma(H); x)$ would equal $1$.

Moreover, we must guarantee that
\begin{itemize}
    \item[(C5)] $\tau(y) = 1$ for all $y < |\sigma|$ with $\sigma(y) = 1$ if $\tau$ is the $(n-1)$-deputy of $\sigma$.
\end{itemize}
For, if $\sigma(y) = 1$ then there could be some $c(n-2,\zeta,y,1)$ active at $\sigma$ which demands $\sigma$ forcing $\Delta_{n-1}(\Gamma(G); y) \downarrow = 1$ and would be replicated when $\tau = d(\sigma)$ is defined. Certainly we do not want to have any $c(n-2,\eta,y,1)$ active at $\tau$ if $\tau(y) = 0$ as we cannot allow $\tau$ forcing $\Delta_{n-1}(\Gamma(G); y) \downarrow = 1 \neq \tau(y)$. But the above condition can be easily met by choosing $\tau = \rho\<1^p0\>$ with $p$ fresh when a $D(\rho)$-module defines a $(n-1)$-deputy.

\subsection{The construction}\label{ss:construction}

Here we formally present how the construction proceeds by stage. To help the reader recall the motivations, some explanations are given in brackets. As mentioned, we recursively assign countably many stages to each requirement so that every stage is assigned to exactly one requirement. At an $R$-stage $s$, the $R$-module acts as below.

(1) $P(\sigma,m)$-module: If there does not exist any $\tau \supseteq \sigma$ with $|\Gamma(\tau)| > m$ then pick a fresh number $l$ and let $\tau = \sigma\<1^l0\>$ and $\Gamma(\tau) = \Gamma(\sigma)\<1^l0\>$.

(2) $D(0,\sigma,\tau)$-modules simply do nothing.

(3) $D(k,\sigma,\tau)$-module where $k > 0$: If $\tau$ is not yet the $k$-deputy of $\sigma$ then proceed to stage $s+1$. Suppose that $\tau$ is the $k$-deputy of $\sigma$. Pick the least $\bar{\nu} \supseteq \tau$ such that \emph{no} $\nu \supseteq \bar{\nu}$ has its $(\tau,k-1)$-deputy defined. Pick a fresh number $l$ and let $\nu = \bar{\nu}\<1^l0\>$. Let the $(\tau,k-1)$-deputy $d(\nu)$ of $\nu$ be $\sigma\<1^{|\nu|}0\>$ and let $\Gamma(d(\nu)) = \Gamma(\nu)\<1^{|\nu|}0\>$. Moreover, for every $Q$-constraint $c(j,\zeta,\vec{x},\bar{x},i)$ with $j < k-1$ and $\tau \subset \zeta \subset \nu$, impose a new $Q$-constraint $c(j,d(\nu),\vec{x},\bar{x},i)$. [The first half defines $(k-1)$-deputies for $\nu$ in a set dense below $\tau$ as extensions of $\sigma$. The second half replicates $Q$-constraints active at $\nu$ to $d(\nu)$, except those active at $\tau$, so that if $\nu$ is expected to force $\Delta_{k-1}(\Gamma(G); \vec{x}) \downarrow = i$ then so is $d(\nu)$.]

(4) $D(\rho)$-module: Pick the least $\bar{\sigma} \supseteq \rho\<0\>$ such that \emph{no} $\sigma \supseteq \bar{\sigma}$ has its $(\rho,n-1)$-deputy defined. Pick a fresh number $l$ and let $\sigma = \bar{\sigma}\<1^l0\>$. Let the $(\rho,n-1)$-deputy $d(\sigma)$ of $\sigma$ be $\rho\<1^{|\sigma|}0\>$ and let $\Gamma(d(\sigma)) = \Gamma(\sigma)\<1^{|\sigma|}0\>$. Moreover, for every $Q$-constraint $c(j,\zeta,\vec{x},\bar{x},i)$ such that $\rho \subset \zeta \subset \sigma$, impose a new $Q$-constraint $c(j,d(\sigma),\vec{x},\bar{x},i)$. [It is similar to $D(k,-)$-modules.]

(5) $Q(0,\sigma,\vec{x})$-module for $n > 1$: If $\Delta(\Gamma(\tau); \vec{x})$ is defined for some $\tau$ comparable with $\sigma$ then proceed to stage $s+1$. Suppose that $\Delta(\Gamma(\tau); \vec{x})$ is undefined for all $\tau$ comparable with $\sigma$. Pick a fresh number $l$ and let $\xi = \Gamma(\sigma\<1^l0\>) = \Gamma(\sigma)\<1^l0\>$. If some $Q$-constraint $c(0,\rho,\vec{x} \uh n-1,\bar{x},i)$ is active at $\sigma$ and $x_{n-1} > \bar{x}$, $x_{n-1}$ being the last element of $\vec{x}$, then let $\Delta(\xi; \vec{x}) = i$. Otherwise let $\Delta(\xi; \vec{x}) = \sigma(x_0)$. [If some $Q$-constraint active at $\sigma$ demands $\Delta_1(\Gamma(G); \vec{x} \uh n-1) \downarrow = i$ being forced then the current $Q$-module obeys this constraint by making $\xi \Vdash \Delta(\Gamma(G); \vec{x}) \downarrow = i$.]

(6) $Q(k,\sigma,\vec{x})$-module where $n-1 > k > 0$: If some $c(k-1,\rho,\vec{x},\bar{x},i)$ is active with $\rho$ and $\sigma$ comparable then proceed to stage $s+1$. Suppose that there is \emph{no} $Q$-constraint as above. Pick a fresh number $l$ and let $\tau = \sigma\<1^l0\>$. If some $Q$-constraint $c(k,\rho,\vec{x} \uh (n-k-1),\bar{x},j)$ is active at $\sigma$ and $x_{n-k-1} > \bar{x}$, $x_{n-k-1}$ being the last element of $\vec{x}$, then let $i = j$, otherwise let $i = \sigma(x_0)$. Impose a new $Q$-constraint $c(k-1,\tau,\vec{x},l,i)$. [If $\sigma$ is expected to force $\Delta_{k+1}(\Gamma(G); \vec{x} \uh (n-k-1)) \downarrow = j$ then the $Q(k,\sigma,\vec{x})$-module obeys the corresponding constraint by imposing a lower layer constraint so that eventually $\tau$ forces $\Delta_{k}(\Gamma(G); \vec{x}) \downarrow = j$.]

(7) $Q(n-1,\sigma,x)$-module: There are two cases below.

(7a) $n = 1$. If $\sigma(x) = 0$ or $\Delta(\Gamma(\rho); x) \downarrow$ for some $\rho$ comparable with $\sigma$ then proceed to stage $s+1$. Suppose that $\sigma(x) = 1$ and $\Delta(\Gamma(\rho); x) \uparrow$ for any $\rho$ comparable with $\sigma$. Pick a fresh number $l$. Define $\xi = \Gamma(\sigma\<1^l0\>) = \Gamma(\sigma)\<1^l0\>$ and $\Delta(\xi; x) = 1$.

(7b) $n > 1$. Firstly pick a fresh number $l$. Act according to the first of the followings that applies
\begin{itemize}
    \item If there already exists a constraint $c(n-2,\zeta,x,-)$ such that $\zeta$ and $\sigma$ are comparable then proceed to stage $s+1$.
    \item If $\sigma(x) = 1$ then impose a constraint $c(n-2,\sigma\<1^l0\>,x,l,1)$. [So the collaboration of $Q$-modules will make $\sigma\<1^l0\>$ forcing $(\forall x_1 > l) \Delta_{n-2}(\Gamma(G); x, x_1) \downarrow = 1$.]
    \item If $\sigma(x) = 0$ and $n = 2$ then let $\xi = \Gamma(\sigma\<1^l0\>) = \Gamma(\sigma)\<1^l0\>$ and let $\Delta(\xi; x,l) = 0$.
    \item If $\sigma(x) = 0$ and $n > 2$ then impose a constraint $c(n-3,\sigma\<1^l0\>,\<x l\>,l,0)$. [It is expected that $\sigma\<1^l0\> \Vdash \Delta_{n-2}(\Gamma(G); x, l) \downarrow = 0$ and thus $\sigma$ does not force $\Delta_{n-1}(\Gamma(G); x) \downarrow = 1$.]
\end{itemize}

\subsection{The verification}\label{ss:verification}

We verify that the construction produces the desired functionals by proving a series of lemmas.

\begin{lemma}\label{lem:codes}
(C1-5) hold.
\end{lemma}

\begin{proof}
It can be proved by an easy induction on stages.
\end{proof}

It is rather easy to see that the $P$'s are satisfied.

\begin{lemma}\label{lem:Gamma}
$\Gamma$ is well-defined and $\Gamma(G)$ is total for every $1$-generic $G$.
\end{lemma}

\begin{proof}
Suppose that at the beginning of stage $s$ we have:
\begin{enumerate}
    \item $\Gamma$ is well-defined, and
    \item $\Gamma(\tau) \supset \Gamma(\sigma)$ whenever $\tau = d(\sigma)$.
\end{enumerate}
Note that $\Gamma$ is extended at stage $s$ only if $s$ is a $P$-, $D(k, -)$- ($k > 0$), $D(\rho)$- or $Q(0,-)$-stage.

If $s$ is a $P(\sigma,m)$-stage then (1) holds at the end of stage $s$, by the induction hypothesis and (C1). Also (2) holds at the end of stage $s$ as well, since $d$ is not changed at stage $s$.

Suppose that $k > 0$ and $s$ is a $D(k,\sigma,\tau)$-stage. Then $\tau = d(\sigma)$. Suppose that the $D(k,\sigma,\tau)$-module picks $\nu \supset \tau$ and defines $d(\nu) = \mu \supset \sigma$. By the induction hypothesis,
$$
    \Gamma(\mu) \supset \Gamma(\nu) \supseteq \Gamma(\tau) \supset \Gamma(\sigma).
$$
So both (1) and (2) hold at the end of stage $s$.

The case for $D(\rho)$-stage is similar to that of $D(k,-)$-stage. And that of $Q(0,-)$-stage is similar to that of $P$-stage. So $\Gamma$ is well-defined.

Clearly every $P(\sigma,m)$ is satisfied and the set
$$
  \{\sigma: |\Gamma(\sigma)| > m\}
$$
is recursive and dense for every $m$. Hence $\Gamma(G)$ is total for every $1$-generic $G$.
\end{proof}

By the proof of Lemma \ref{lem:Gamma}, we have the following fact of deputies.

\begin{lemma}\label{lem:d-Gamma}
If $\tau = d(\sigma)$ then $\Gamma(\tau) \supset \Gamma(\sigma)$.
\end{lemma}

Next we turn to the $Q$'s. It is important to learn how deputies and $Q$-constraints interact.

\begin{lemma}\label{lem:replication}
Suppose that $\tau_k$ is the $k$-deputy of $\sigma_k$ and $\rho$ is the greatest common initial segment of $\sigma_k$ and $\tau_k$.
\begin{enumerate}
    \item If $k = n-1$ then every $Q$-constraint $c(m,\eta,-)$ active at $\tau_k$ is either active at $\rho$ or a replica of some $c(m,\zeta,-)$ with $\rho \subset \zeta \subset \sigma_k$.
    \item If $k < n-1$ then there exist $\sigma_{k+1}$ and $\tau_{k+1}$ such that
    \begin{enumerate}
    	\item $\sigma_{k+1} \subset \sigma_k$ and $\tau_{k+1} \subset \tau_k$;
    	\item $\sigma_{k+1}$ is the $(k+1)$-deputy of $\tau_{k+1}$ and $\tau_k$ is the $(\sigma_{k+1},k)$-deputy of $\sigma_k$;
    	\item every $Q$-constraint $c(m,\eta,-)$ active at $\tau_k$ is either active at $\tau_{k+1}$ or a replica of some $c(m,\zeta,-)$ with $\sigma_{k+1} \subset \zeta \subset \sigma_k$.
    \end{enumerate}
    \item If $m < k$ then for every $c(m,\eta,\vec{x},\bar{x},i)$ active at $\tau_k$ there exists $c(m,\zeta,\vec{x},\bar{x},i)$ active at $\sigma_k$ and for every $c(m,\mu,\vec{y},\bar{y},j)$ active at $\sigma_k$ there exists $c(m,\nu,\vec{y},\bar{y},j)$ active at $\tau_k$.
\end{enumerate}
\end{lemma}

\begin{proof}
(1) and (2) are trivial facts of the construction.

For $k = n-1$, (3) follows from the fact that $\tau_k = d(\sigma_k)$ is defined by the $D(\rho)$-module and every $Q$-constraint active at $\sigma_k$ either is active at $\rho$ or has a replica active at $\tau_k$. Suppose that $n-1 > k > m$. There exist $\sigma_{k+1}$ and $\tau_{k+1}$ as in (2). If $c(m,\eta,-)$ is active at $\tau_{k}$ then by (2c) above it is either active at $\tau_{k+1}$ or a replica of some $c(m,\zeta,-)$ active at $\sigma_k$. So the conclusion either follows from the induction hypothesis or holds trivially. For $c(m,\mu,-)$ active at $\sigma_k$, if $\sigma_{k+1} \subset \mu$ then it has a replica $c(m,\tau_k,-)$; otherwise it is active at $\sigma_{k+1}$ and the conclusion follows from the induction hypothesis.
\end{proof}

The following fact implies that $Q$-constraints are arranged consistently.

\begin{lemma}\label{lem:Q-constraints-cons-1}
For each $k, \vec{x},\mu$ there exists at most one $c(k,\zeta,\vec{x},\bar{x},i)$ active at $\mu$.
\end{lemma}

\begin{proof}
Suppose that the lemma holds at the end of stage $s-1$ and at least one new constraint is imposed at stage $s$. We prove that the lemma also holds at the end of stage $s$. Since new $Q$-constraints are imposed only at $Q$- or $D$-stages, there are three cases.

(i) $s$ is a $Q$-stage. The conclusion is guaranteed by the construction.

(ii) $s$ is a $D(k',\sigma,\tau)$-stage. Then $\tau$ is the $k'$-deputy of $\sigma$. Suppose that the $D$-module picks $\bar{\nu} \supset \tau$ and $\nu = \bar{\nu}\<1^p0\>$ and defines $\mu = d(\nu) = \sigma\<1^q0\>$. It suffices to prove that if $c = c(k,\zeta,\vec{x},\bar{x},i)$ and $c' = c(k,\eta,\vec{x},\bar{y},j)$ are active at $\mu$ then $\zeta = \eta$, $\bar{x} = \bar{y}$ and $i = j$. If neither $\zeta$ nor $\eta$ equals $\mu$ then by the construction both constraints are imposed before stage $s$ and thus coincide by the induction hypothesis. Assume that $\zeta = \mu$ and $c$ is imposed at stage $s$. So $k'-1 > k$ and $c$ is a replica of some $c_0 = c(k,\zeta_0,-)$ active at $\nu$ and imposed before stage $s$. If $c'$ is active at $\sigma$, then by (3) of Lemma \ref{lem:replication} and that $k' > k$ there exists $c'_1 = c(k,\rho,\vec{x},\bar{y},j)$ active at $\tau$. So both $c_0$ and $c'_1$ are active at $\nu$ and imposed before stage $s$. Hence $c_0$ and $c'_1$ coincide by the induction hypothesis. However, this coincidence implies that $c_0$ is active at $\tau$ and has no replica imposed at stage $s$ by the $D(k',\sigma,\tau)$-module and thus contradicts $c$ being a replica of $c_0$. If $c'$ is not active at $\sigma$ then $c'$ is a replica of some $c'_0 = c(k,\eta_0,-)$ by (2b) of Lemma \ref{lem:replication}. So both $c_0$ and $c_0'$ are active at $\nu$ and imposed before stage $s$. By the induction hypothesis again, $c_0$ and $c_0'$ coincide. Then $c$ and $c'$ are both replicas of the same constraint imposed at stage $s$ and thus coincide as well.

(iii) $s$ is a $D(\rho)$-stage. The proof for this case is similar to that for (ii).
\end{proof}

\begin{lemma}\label{lem:Delta}
$\Delta$ is well-defined.
\end{lemma}

\begin{proof}
Lemma \ref{lem:Q-constraints-cons-1} implies that a $Q(0,-)$ module never encounters conflicting $c(0,-)$'s. So the validity of $\Delta$ follows from (C2).
\end{proof}

A constraint imposed before or at stage $s$ is a constraint \emph{active} at stage $s$. An active constraint is one active at some stage. If some $c(k,\zeta,\vec{x},\bar{x},i)$ is active then we expect that $\zeta$ forces $\Delta_k(\Gamma(G); \vec{x}, x) \downarrow = i$ for all $x > \bar{x}$. The next lemma is important for this end.

\begin{lemma}\label{lem:Q-thresholds}
Suppose that $c(k,\zeta,\vec{x},\bar{x},i)$ is active, $\zeta$ and $\eta$ are comparable and $x > \bar{x}$.
\begin{enumerate}
    \item If there also exists an active $c(k-1,\eta,\vec{x}\<x\>,\bar{y},j)$ then $\eta \supseteq \zeta$.
    \item If $k = 0$, $\Delta(\Gamma(\eta); \vec{x},x) \downarrow$ but $\Delta(\Gamma(\eta \uh (|\eta| - 1)); \vec{x},x) \uparrow$ then $\eta \supseteq \zeta$.
\end{enumerate}
\end{lemma}

\begin{proof}
(1) Suppose that $\zeta$ is the shortest witnessing the failure of clause (1), i.e., there exist active $c_k = c(k,\zeta,\vec{x},\bar{x},i)$ and $c_{k-1} = c(k-1,\eta,\vec{x}\<x\>,\bar{y},j)$ and $x > \bar{x}$ but $\eta \subset \zeta$. Let $s$ and $s_0$ be the stages at which $c_k$ and $c_{k-1}$ are imposed respectively. Then $s \geq s_0$, since constraints are always imposed at strings of fresh length. If $s = s_0$ then $s$ must be a $D$-stage as at a $Q$-stage at most one constraint is imposed. Then $\eta = \zeta$ by the $D$-modules, contradicting the choice of $\zeta$. Hence $s > s_0$.

If $s$ is a $Q$-stage then $\bar{x}$ is greater than any number appeared in the construction before stage $s$. In particular $\bar{x} > x$, contradicting the premise.

Suppose that $s$ is a $D(k',\sigma,\tau)$-stage. Then $\tau$ is the $k'$-deputy of $\sigma$ and $\zeta$ is the $(k'-1)$-deputy of some $\nu \supset \tau$. Moreover, $k'-1 > k$ and $c_k$ is a replica of some $c_k' = c(k,\zeta',\vec{x},\bar{x},i)$ such that $\tau \subset \zeta' \subset \nu$. As $c_{k-1}$ is imposed at stage $s_0 < s$, $c_{k-1}$ is active at $\sigma$ by (2) of Lemma \ref{lem:replication}. By (3) of Lemma \ref{lem:replication}, there exists $c_{k-1}' = c(k-1,\eta',\vec{x}\<x\>,\bar{y},j)$ active at $\tau$. But then $\zeta'$ is shorter than $\zeta$ and witnesses the failure of clause (1). So we have a desired contradiction.

The case that $s$ is a $D(\rho)$-stage is similar.

(2) Assume that $\zeta$ is the shortest witnessing the failure of clause (2). Let $c_0$ denote $c(0,\zeta,\vec{x},\bar{x},i)$. Let $s$ be the stage when $c_0$ is imposed. By the premise of (2), $\Gamma(\eta)$ is explicitly defined at some stage $s_0$. By the construction, $\eta$ is longer than any $\xi$ appeared before stage $s_0$. As we assume that $\zeta \supset \eta$, $s_0 \leq s$.

If $s$ is a $Q$-stage then $s_0 < s$, since a $Q$-module cannot extend $\Gamma$ when it imposes a constraint. So $\bar{x}$ is greater than any number appeared before stage $s$, in particular $\bar{x} > x$, contradicting the premise.

Suppose that $s$ is a $D(k',\sigma,\tau)$-stage. Then $\tau$ is the $k'$-deputy of $\sigma$ and $\zeta$ being the $(k'-1)$-deputy of some $\nu \supset \tau$ is of the form $\sigma\<1^l0\>$. Moreover, $c_0$ is a replica of some $c_0' = c(0,\zeta',-)$ such that $\tau \subset \zeta' \subset \nu$. By the construction, $|\zeta| > |\zeta'|$. As $\sigma \subset \rho \subset \zeta$ implies that $\Gamma(\rho) = \Gamma(\sigma)$ by (C1), $\eta \subseteq \sigma$. Since $\Gamma(\tau) \supset \Gamma(\sigma) \supseteq \Gamma(\eta)$, $\Delta(\Gamma(\tau); \vec{x}, x) \downarrow$. Let $\eta'$ be the shortest initial segment of $\tau$ such that $\Delta(\Gamma(\eta'); \vec{x}, x) \downarrow$. Then $\zeta'$ witnesses the failure of clause (2) as well, contradicting $\zeta$ being the shortest.

By a similar argument, it is impossible that $s$ is a $D(\rho)$-stage.
\end{proof}

Below we show that every active constraint is obeyed.

\begin{lemma}\label{lem:Q-constraints-cons-2}
Suppose that $c(k,\xi,\vec{x},\bar{x},i)$ is active, $\xi$ is comparable with $\mu$ and $x > \bar{x}$.
\begin{enumerate}
    \item If $k > 0$ and $c(k-1,\mu,\vec{x} \<x\>,\bar{y},j)$ is also active then $j = i$.
    \item If $k = 0$ and $\Delta(\Gamma(\mu); \vec{x},x) \downarrow = j$ then $j = i$.
\end{enumerate}
\end{lemma}

\begin{proof}
Assume that both (1) and (2) hold before stage $s$. Fix parameters as in the premises. Let $c_k = c(k,\xi,\vec{x},\bar{x},i)$. By Lemma \ref{lem:Q-thresholds}, we may assume that $\mu \supseteq \xi$.

(1) Assume that $c_k$ is imposed at stage $s_0$ and $c_{k-1} = c(k-1,\mu,\<\vec{x} x\>,\bar{y},j)$ at stage $s$. Then $s \geq s_0$ and $s$ is a $Q$- or $D$-stage.

If $s$ is a $Q$-stage then $\xi \subset \mu$ as at each $Q$-stage at most one constraint is imposed. If $c_{k-1}$ is imposed by a $Q(k,\sigma,-)$-module then $\mu$ is of the form $\sigma\<1^l0\>$ and $c_k$ is active at $\sigma$ by (C4). So $j = i$ by the $Q(k,\sigma,-)$-module and Lemma \ref{lem:Q-constraints-cons-1}.

Suppose that $s$ is a $D(k',\sigma,\tau)$-stage. Then $\tau$ is the $k'$-deputy of $\sigma$, $\sigma \subset \mu$ and $\mu$ is the $(k'-1)$-deputy of some $\nu \supset \tau$. Moreover, $k'-1 > k-1$ and $c_{k-1}$ is a replica of some $c'_{k-1} = c(k-1,\eta,-)$ such that $\tau \subset \eta \subset \nu$. If $\xi \subset \mu$, $\xi \subseteq \sigma$. By (3) of Lemma \ref{lem:replication}, there exists $c_k' = c(k,\zeta,\vec{x},\bar{x},i)$ active at $\tau$. If $\xi = \mu$ then $c_k$ is also a replica of some $c_k'' = c(k,\zeta',-)$ active at $\nu$. In either case, $j = i$ by the induction hypothesis.

The case where $s$ is a $D(\rho)$-stage is similar.

(2) We may assume that $\Delta(\Gamma(\mu \uh (|\mu| - 1)); \vec{x},x) \uparrow$. So $\Gamma(\mu)$ is explicitly defined in the construction at some stage $s$.

If $s$ is a $P$-stage then $\Delta(\Gamma(\mu); \vec{x},x)$ diverges too as $\Gamma(\mu)$ is of a fresh length and $P$-modules never change $\Delta$.

If $s$ a $Q(0,\sigma,-)$-stage then $\mu$ is of the form $\sigma\<1^l0\>$ with some fresh $l$. As $Q(0,-)$-modules never impose constraints, $\xi \subset \mu$. By (C4), $c_k$ is active at $\sigma$. So by Lemma \ref{lem:Q-constraints-cons-1} and the $Q(0,-)$-module, $j = i$.

Suppose that $s$ is a $D(k',\sigma,\tau)$-stage. Then $k' > 0 = k$ and $\mu = d(\nu) = \nu\<1^l0\>$ for some $l$ and $\nu \supset \tau$. So $\Gamma(\mu) \supset \Gamma(\nu)$. As the $D(k',\sigma,\tau)$-module does not extend $\Delta$ explicitly, $\Delta(\Gamma(\nu); \vec{x},x) \downarrow = j$ as well. If $c_k$ is active at $\sigma$, then there exists $c_k' = c(k,\zeta,\vec{x},\bar{x},i)$ active at $\tau$ by (3) of Lemma \ref{lem:replication}. If $c_k$ is not active at $\sigma$ then $c_k$ is a replica of some $c_k''$ active at $\nu$. In either case, $j = i$ by the induction hypothesis applying to $c_k'$ and $\nu$ or $c_k''$ and $\nu$.

The case where $s$ is a $D(\rho)$-stage is similar.
\end{proof}

Now we can prove the satisfaction of the $Q$'s for $n = 1$.

\begin{lemma}\label{lem:Q-top-1}
If $n = 1$ and $x < |\sigma|$ then $Q(n-1,\sigma,x)$ is satisfied. Hence
$$
    (\forall x)(x \in G \leftrightarrow \Delta(\Gamma(G); x) \downarrow = 1)
$$
for every $1$-generic $G$.
\end{lemma}

\begin{proof}
Suppose that $x < |\sigma|$ and $\sigma(x) = 1$. By the $Q(0,\sigma,x)$-module there exists $\tau \supset \sigma$ with $\Delta(\Gamma(\tau); x) \downarrow = 1$.

For the other direction, assume that $\sigma$ is the shortest such that
\begin{equation}\label{eq:Q-top-1-contradiction}
    (\exists x)(\sigma(x) = 0 \wedge \Delta(\Gamma(\sigma); x) \downarrow = 1).
\end{equation}
Fix an $x$ satisfying the matrix of \eqref{eq:Q-top-1-contradiction}. We may assume that $\Delta(\Gamma(\sigma \uh (|\sigma| - 1)); x) \uparrow$. So $\Gamma(\sigma)$ is explicitly defined in the construction at some stage $s$. If $s$ is a $P$-stage then $\sigma = \rho\<1^l0\>$ for some fresh $l$ and no new computation of $\Delta$ is enumerated at stage $s$. Thus \eqref{eq:Q-top-1-contradiction} cannot hold by the choice of $\sigma$. If $s$ is a $Q(0,-,x)$-stage then by the construction nothing happens at stage $s$. Suppose that $s$ is a $D(\rho)$-stage. Then $\sigma$ is the deputy of some $\tau \supseteq \rho\<0\>$, $\sigma = \rho\<1^{|\tau|}0\>$ and $\Delta(\Gamma(\sigma)) = \Delta(\Gamma(\tau))$. So $\Delta(\Gamma(\sigma); |\sigma|-1)$ is undefined and $x < |\sigma|-1$. As $\sigma(x) = 0$, $x < |\rho|$. So $\tau(x) = \rho(x) = 0$ and $\Delta(\Gamma(\tau); x) = 1$, contradicting the choice of $\sigma$. Since $n = 1$, there exist no $D(k,-)$-stages with $0< k < n$.

So $Q(0,\sigma,x)$ is satisfied. The other half follows easily from the first half.
\end{proof}

For $n > 1$, the satisfaction of the $Q$'s follows from the next two lemmata.

\begin{lemma}\label{lem:Q-forcing}
Suppose that $n-1 > k$. Then every $Q(k,-)$ is satisfied, and if $c(k,\sigma,\vec{x},\bar{x},i)$ is active then $\Delta_{k+1}(\Gamma(G); \vec{x}) = i$ for each $1$-generic $G \supset \sigma$ and thus
$$
    \sigma \Vdash \Delta_{k+1}(\Gamma(G); \vec{x}) \downarrow = i.
$$
\end{lemma}

\begin{proof}
We prove the lemma by induction on $k$.

Suppose that $k = 0$. By the construction, for every $\vec{x}\<x\>$ of length $n$ the following set is recursive and dense
$$
    \{\sigma: \Delta(\Gamma(\sigma); \vec{x},x) \downarrow\}.
$$
So every $Q(0,\sigma,\vec{x}\<x\>)$ is satisfied. Furthermore, assume that $c(0,\sigma,\vec{x},\bar{x},i)$ is active and $\Delta(\Gamma(\mu); \vec{x}, x) \downarrow = j$ where $x > \bar{x}$ and $\mu$ is comparable with $\sigma$. By (2) of Lemma \ref{lem:Q-constraints-cons-2}, $j = i$. Hence
$$
    \sigma \Vdash \Delta_1(\Gamma(G); \vec{x}) \downarrow = i.
$$

Suppose that $k > 0$. By the construction, for every $\vec{x}\<x\>$ of length $n-k$ the following set is recursive and dense
$$
    \{\tau: \text{ there exists an active } c(k-1,\tau,\vec{x}\<x\>,\bar{x},j)\}.
$$
By the induction hypothesis,
$$
    \tau \Vdash \Delta_k(\Gamma(G); \vec{x},x) \downarrow
$$
for densely many $\tau$. So the $Q(k,-)$'s are satisfied. Furthermore, assume that $c(k,\sigma,\vec{x},\bar{x},i)$ is active and $G$ is a $1$-generic extending $\sigma$. Then for every $x > \bar{x}$ there exists $\mu \subset G$ such that $c(k-1,\mu,\vec{x}\<x\>,\bar{y},j)$ is active. By (1) of Lemma \ref{lem:Q-constraints-cons-2}, $j = i$. By the induction hypothesis, $\Delta_k(\Gamma(G); \vec{x}, x) \downarrow = i$ and $\Delta_{k+1}(\Gamma(G); \vec{x}, x) = \lim_x \Delta_k(\Gamma(G); \vec{x}, x) = i$. Hence $\sigma \Vdash \Delta_{k+1}(\Gamma(G); \vec{x}) \downarrow = i$.
\end{proof}

\begin{lemma}\label{lem:Q-top-n}
If $n > 1$ and $x$ and $\sigma$ are such that $|\sigma| > x$ then
$$
    \sigma(x) = 1 \to \sigma \Vdash \Delta_{n-1}(\Gamma(G); x) \downarrow = 1
$$
and
$$
    \sigma(x) = 0 \to \sigma \Vdash \neg (\Delta_{n-1}(\Gamma(G); x) \downarrow = 1).
$$
Hence for $n > 1$ every $Q(n-1,\sigma,x)$ is satisfied.
\end{lemma}

\begin{proof}
Suppose that $\sigma(x) = 1$. By the $Q(n-1,-,x)$-modules, the following set is recursive and dense below $\sigma$:
$$
    \{\mu: \text{some } c(n-2,\zeta,x,\bar{x},i) \text{ is active at } \mu\}.
$$
Since no module can ever impose any $c(n-2,\xi,y,\bar{y},0)$, the dense set above equals to the set below
$$
    \{\mu: \text{some } c(n-2,\zeta,x,\bar{x},1) \text{ is active at } \mu\}.
$$
By Lemma \ref{lem:Q-forcing},
$$
    \sigma \Vdash \Delta_{n-1}(\Gamma(G); x) \downarrow = 1
$$

On the other hand, suppose that $\sigma(x) = 0$.

\begin{claim}\label{clm:Q-top-n}
There is no active $c(n-2,\zeta,x,\bar{x},i)$ such that $\zeta$ and $\sigma$ are comparable.
\end{claim}

\begin{proof}[Proof of Claim \ref{clm:Q-top-n}]
For a contradiction, assume that there exists such an active $c(n-2,\zeta,x,\bar{x},i)$. Since the construction never produces a constraint like $c(n-2,-,0)$, $i = 1$. By the construction, $|\zeta| > x$. So we may assume that $\zeta = \sigma$ and $\sigma$ is the shortest with such a constraint. Clearly such a $Q$-constraint cannot be imposed by a $Q$-module. Neither can it be imposed by a $D(k,-)$-module with $k \leq n-1$, since such a module can only replicate $c(j,-)$ with $j < k-1 \leq n-2$. Suppose that a $D(\rho)$-module imposes $c$. Then $\sigma \supset \rho\<1\>$ and $\sigma$ is the deputy of some $\tau \supset \rho\<0\>$. By (3) of Lemma \ref{lem:replication}, there exists some $c(n-2,\eta,x,\bar{x},1)$ active at $\tau$. As $\tau(y) = 1$ implies $\sigma(y) = 1$ for all $y$, $\sigma(x) = 0$ implies $\eta(x) = \tau(x) = 0$ too, contradicting the choice of $\sigma$.
\end{proof}

Let $\mu \supset \sigma$ be arbitrary. By Claim \ref{clm:Q-top-n}, no $c(n-2,\zeta,x,\bar{x},i)$ is active at $\mu$. Note that there are at most finitely many $Q$-constraints active at $\mu$. By the $Q(n-1,\mu,x)$-module, there exist $x_1$ and $\zeta \supset \mu$, such that either $n = 2$ and $\Delta(\Gamma(\zeta); x,x_1) \downarrow = 0$ or $n > 2$ and there exists an active constraint $c(n-3,\zeta,\<x x_1\>,\bar{x},0)$. By Lemma \ref{lem:Q-forcing}, either case implies that for each $\bar{x}$ the following set is dense below $\sigma$
$$
    \{\tau: (\exists x_1 > \bar{x}) \tau \Vdash \Delta_{n-2}(\Gamma(G); x, x_1) \downarrow = 0\}.
$$
Hence
$$
    \sigma \Vdash \neg (\Delta_{n-1}(\Gamma(G); x) \downarrow = 1).
$$

It follows that $Q(n-1,\sigma,x)$ is satisfied.
\end{proof}

By Lemma \ref{lem:Q-forcing} and that the following set is dense and recursive
$$
    \{\sigma: \text{ some } c(n-k-1,\sigma,\vec{x},x,i) \text{ is active}\},
$$
$\Delta_{n-k}(\Gamma(G))$ is $G$-recursive for $n \geq k > 1$ and every $1$-generic $G$. Then by Lemma \ref{lem:Q-top-n}, every $1$-generic $G$ is $\Sigma^0_{k}$ in $\Delta_{n-k}(\Gamma(G))$.

Finally, we prove the satisfaction of the $D$'s and also (A2).

\begin{lemma}\label{lem:D-satisfaction-1}
If $\tau$ is defined as the $k$-deputy of $\sigma$ then $D(k,\sigma,\tau)$ is satisfied.
\end{lemma}

\begin{proof}
We prove the lemma by induction on $k$.

By Lemma \ref{lem:d-Gamma}, the lemma holds for $k = 0$.

Suppose that $k > 0$ and $\tau$ is defined as the $k$-deputy of $\sigma$. By the $D(k,\sigma,\tau)$-module, the set below is dense below $\tau$
$$
    \{\nu: (\exists \mu \supset \sigma)(\mu \text{ is the $(k-1)$-deputy of } \nu)\}.
$$
By the induction hypothesis, the following set is also dense below $\tau$
$$
    \{\nu: (\exists \mu \supset \sigma)(D(k-1,\nu,\mu) \text{ is satisfied})\}.
$$
So $D(k,\sigma,\tau)$ is satisfied.
\end{proof}

\begin{lemma}\label{lem:D-satisfaction-2}
Every $D(\rho)$ is satisfied. Hence $G$ is not $\Pi^0_n$ in $\Gamma(G)$ for every $n$-generic $G$.
\end{lemma}

\begin{proof}
By the $D(\rho)$-module, the set below is dense below $\rho\<0\>$
$$
    \{\sigma: (\exists \tau \supset \rho\<1\>)(\tau \text{ is the $(n-1)$-deputy of } \sigma)\}.
$$
By Lemma \ref{lem:D-satisfaction-1}, the following set is also dense below $\rho\<0\>$
$$
    \{\sigma: (\exists \tau \supset \rho\<1\>)(D(n-1,\sigma,\tau) \text{ is the satisfied})\}.
$$
Hence $D(\rho)$ is satisfied.

Suppose that $G$ is $n$-generic. Then $\Gamma(G)$ is total by Lemma \ref{lem:Gamma}. Let $X = \Gamma(G)$. Assume that $\omega - G$ equals
$$
    S = \{x: (\exists y) \varphi(x,y)\}
$$
for some $\varphi$ which is $\Pi^0_{n-1}$ in $X$. For each $x \not\in G$ let $\rho_x = G \uh x$. By the $n$-genericity of $G$ and the satisfaction of $D(\rho_x)$, there exists an initial segment $\sigma_x$ of $G$ such that $\rho_x \subset \sigma_x$, $\sigma_x \Vdash \varphi(x,y)$ for some $y$ and $D(n-1,\sigma_x,\tau_x)$ is satisfied for some $\tau_x \supset \rho\<1\>$. By Lemma \ref{lem:good-deputy}, $\tau_x \Vdash \varphi(x,y)$. So the following $\Sigma^0_n$ set is dense along $G$
$$
    D = \{\tau: (\exists x < |\tau|)(\tau(x) = 1 \wedge \tau \Vdash (\exists y) \varphi(x,y))\}.
$$
As $G$ is $n$-generic, it has an initial segment $\sigma$ in $D$. So $\omega - G \neq S$.
\end{proof}

So we prove Theorem \ref{thm:Main}. Moreover, by the last two lemmata and the remark preceding Lemma \ref{lem:D-satisfaction-1}, every $n$-generic $G$ is properly $\Sigma^0_k$ in $\Delta_{n-k}(\Gamma(G))$ for $n \geq k > 1$.

\section{Remarks}\label{s:Remarks}

Recall that for positive $n$ a set $X$ is \emph{generalized ${\rm low}_n$} ($GL_n$ for short) if
$$
    X^{(n)} \equiv_T (X \oplus \emptyset')^{(n-1)}.
$$
Every $n$-generic $G$ satisfies $G^{(n)} \equiv_T G \oplus \emptyset^{(n)}$ and thus is $GL_n$ (\cite[Lemma 2.6]{Jockusch:1980.generic}). Jockusch proved that every $2$-generic bounds some $X$ which is $GL_2$ but not $GL_1$ as a corollary of the relative recursive enumerability of $1$-generics (\cite[Corollary 5.5]{Jockusch:1980.generic}). Similarly he also proved that every $3$-generic bounds a set in $GL_3 - GL_2$ (\cite[Corollary 5.11]{Jockusch:1980.generic}). Then Jockusch conjectured that every $(n+1)$-generic bounds a set in $GL_{n+1} - GL_{n}$ for all positive $n$ (\cite[Conjecture 5.14]{Jockusch:1980.generic}). Below we confirm this conjecture.

Firstly, by relativizing the corresponding proof we have an analog of Lemma \ref{lem:good-deputy}.

\begin{lemma}\label{lem:good-deputy-rel}
Let $Y$ be fixed. If $D(n-1,\sigma,\tau)$ is satisfied and $\sigma \Vdash \varphi$ for some $\varphi$ which is $\Pi^0_{n-1}$ in $\Gamma(G) \oplus Y$ then $\tau \Vdash \varphi$ as well.
\end{lemma}

\begin{corollary}\label{cor:GL}
For $n > 0$, every $(n+1)$-generic bounds a set in $GL_{n+1} - GL_{n}$.
\end{corollary}

\begin{proof}
Let $\Gamma$ be the functional constructed in \S \ref{s:Main-theorem} for $n$ and let $G$ be an arbitrary $(n+1)$-generic. Then $\Gamma(G)$ is total and $G$ is $\Sigma^0_n$ in $\Gamma(G)$.

Let $X = \Gamma(G)$. We claim that $G$ is not $\Pi^0_n$ in $X \oplus \emptyset'$. For a contradiction, assume that $\varphi$ is $\Pi^0_{n-1}$ in $X \oplus \emptyset'$ and the complement of $G$ equals the following set
$$
    S = \{x: (\exists y) \varphi(x,y)\}.
$$
For each $x \not\in G$, let $\rho_x = G \uh x$ and let $\sigma_x$ be an extension of $\rho_x \<0\>$ forcing $\varphi(\Gamma(G),x,y)$ for some $y$. As $\varphi$ is $\Pi^0_{n-1}$ in $X \oplus \emptyset'$, $\varphi$ is $\Pi^0_n$ in $G$ for $n > 1$ or $\Delta^0_2$ in $G$ for $n = 1$. By the $(n+1)$-genericity of $G$, $\sigma_x$ exists. By the construction in \S \ref{s:Main-theorem}, we may assume that $D(n-1,\sigma_x,\tau_x)$ is satisfied for some $\tau_x \supset \rho_x \<1\>$. By Lemma \ref{lem:good-deputy-rel}, $\tau_x \Vdash \varphi(x,y)$ too. So the set below is dense along $G$
$$
    T = \{\tau: (\exists x < |\tau|)(\tau(x) = 1 \wedge \tau \Vdash (\exists y) \varphi(x,y))\}.
$$
Note that $T$ is $\Sigma^0_{n+1}$. By the $(n+1)$-genericity of $G$, there exists a finite initial segment of $G$ in $T$. Thus $G \cap S \neq \emptyset$, giving us the desired contradiction.

So $G \leq_T X^{(n)}$ but $G \not\leq_T (X \oplus \emptyset')^{(n-1)}$. Hence $X \not\in GL_n$. On the other hand,
$$
    X^{(n+1)} \leq_T G^{(n+1)} \equiv_T G \oplus \emptyset^{(n+1)} \leq_T (X \oplus \emptyset')^{(n)} \leq_T X^{(n+1)}.
$$
Thus $X \in GL_{n+1}$.
\end{proof}

The above proof is an application of deputies. As another application, we obtain the following corollary strengthening Theorem \ref{thm:Main} and also \cite[Theorem 3]{Anderson:2011} of Anderson. Let us recall some notions before stating the corollary. Let $(D_i: i \in \omega)$ be a fixed enumeration of all finite sets such that $x \in D_i$ and $|D_i|$ are uniformly recursive in both $x$ and $i$. A \emph{strong array} is a pairwise disjoint subsequence of $(D_i: i \in \omega)$. A set $X$ is \emph{hyperimmune} in another set $Y$ if for any $Y$-recursive strong array $(D_{k_i}: i \in \omega)$ there exists $i$ such that $X \cap D_{k_i} = \emptyset$.

\begin{corollary}\label{cor:generic-himmune}
If $n > 0$ and $G$ is $n$-generic then there exists $X \leq_T G$ such that $G$ is $\Sigma^0_n$ in $X$ and the complement of $G$ is hyperimmune in $X^{(n-1)}$.
\end{corollary}

\begin{proof}
Fix $n > 0$. Let $\Gamma$ be the functional constructed in \S \ref{s:Main-theorem}. Then $X = \Gamma(G)$ is recursive in $G$ and $G$ is $\Sigma^0_n$ in $X$.

For a contradiction, assume that $\Phi$ is a Turing machine such that $f = \Phi(X^{(n-1)})$ is total and $(D_{f(i)}: i \in \omega)$ is a strong array witnessing the complement of $G$ being non-hyperimmune in $X^{(n-1)}$. Let $\varphi(i,x)$ be a formula asserting that $f(i) = x$ which is $\Sigma^0_n$ in $X$. Write $\varphi(i,x)$ as $(\exists y) \psi(i,x,y)$ where $\psi$ is $\Pi^0_{n-1}$ in $\Gamma(G)$. For each $m$, let $\rho_m = G \uh m$ and let $i_m$ and $k_m$ be such that $f(i_m) = k_m$ and $m < \min D_{k_m}$. By the $n$-genericity of $G$, there exist $\sigma_m \subset G$ and $y_m$ such that $\max D_{k_m} < |\sigma_m|$ and $\sigma_m$ forces $\psi(i_m,k_m,y_m)$.
Let $l_m$ be the least $l$ such that $l \in D_{k_m}$ and $\sigma_m(l) = 0$. As $m < \min D_{k_m}$, $l_m > m$ and $\rho_m \subset \sigma_m \uh l_m$. By the construction in \S \ref{s:Main-theorem}, we may assume the existence of some $\tau_m \supset (\sigma_m \uh l_m) \<1\>$ such that $D(n-1,\sigma_m,\tau_m)$ is satisfied and $\tau_m(x) = 1$ for all $x \in D_{k_m}$. By Lemma \ref{lem:good-deputy}, $\tau_m \Vdash \psi(i_m,k_m,y_m)$ and thus $\tau_m \Vdash \varphi(i_m,k_m)$. So the following $\Sigma^0_n$ set is dense along $G$
$$
    T = \{\tau: (\exists i,k) \tau \Vdash (D_k \subset G \wedge \Phi(\Gamma(G); i) \downarrow = k)\}.
$$
By the $n$-genericity of $G$, $G$ meets $T$, i.e., there exists an initial segment of $G$ in $T$. So $D_{f(i)} \subset G$ for some $i$ and we have the desired contradiction.
\end{proof}

Kumabe \cite{Kumabe:1991} obtained a strengthening of the relative recursive enumerability of $1$-generics that for every positive $n$ and every $n$-generic $G$ there exists a $G$-recursive $n$-generic $H$ such that $G$ is recursively enumerable but not recursive in $H$. So it is natural to ask whether Theorem \ref{s:Main-theorem} allows such a stronger version. More precisely,

\begin{question}
For positive $n$, is every $n$-generic $G$ properly $\Sigma^0_n$ in another $n$-generic $H$ which is recursive in $G$?
\end{question}

As mentioned, the r.e.a.\ degrees form a large class. Let ${\bf R}_1$ denote this class, and in general let ${\bf R}_n$ denote the class whose definition is obtained from replacing recursive enumerability by being $\Sigma^0_n$ in the definition of $R$. Kurtz \cite{Kurtz:81.thesis} (see also \cite[Theorem 8.21.8]{Downey.Hirschfeldt:2010.book}) proved that ${\bf R}_1$ includes the degrees of $2$-random reals. So we may ask whether a parallel property holds for $n$-randoms with $n > 2$.

\begin{question}
For $n > 2$, is every $n$-random $X$ properly $\Sigma^0_{n-1}$ in some $X$-recursive $Y$? Does ${\bf R}_{n-1}$ include the $n$-random degrees?
\end{question}

\bibliographystyle{plain}

\end{document}

%% file: diagram-deputies.tex
\[\xygraph{!{<0mm,0mm>;<1mm,0mm>:<0mm,1mm>::}
!{(45,0)}="root"
!{(45,10)}*=0{\bullet}="rho"
!{(42,7)}*=0{\rho}
!{(25,20)}="rho0"
!{(60,20)}="rho1"
!{(20,25)}="sigma"
!{(15,30)}*=0{\bullet}="sigma0"
!{(12,27)}*=0{\sigma_0}
!{(25,40)}*=0{\bullet}="sigma1"
!{(21,40)}*=0{\sigma_1}
!{(10,40)}="sigma1.5"
!{(5,50)}*=0{\bullet}="sigma2"
!{(2,47)}*=0{\sigma_2}
!{(20,55)}="mu"
!{(15,65)}*=0{\bullet}="mu0"
!{(12,62)}*=0{\mu_0}
!{(30,70)}*=0{\bullet}="mu1"
!{(26,70)}*=0{\mu_1}
!{(55,35)}*=0{\bullet}="tau0"
!{(52,32)}*=0{\tau_0}
!{(70,35)}="tau0.5"
!{(65,45)}*=0{\bullet}="tau1"
!{(62,45)}*=0{\tau_1}
!{(75,55)}*=0{\bullet}="tau2"
!{(71,55)}*=0{\tau_2}
!{(50,60)}*=0{\bullet}="nu0"
!{(46,60)}*=0{\nu_0}
!{(60,65)}*=0{\bullet}="nu1"
!{(56,65)}*=0{\nu_1}
!{(120,0)}="groot"
!{(120,10)}*=0{\bullet}="grho"
!{(115,7)}*=0{\Gamma(\rho)}
!{(110,20)}*=0{\bullet}="gsigma0"
!{(105,17)}*=0{\Gamma(\sigma_0)}
!{(130,25)}*=0{\bullet}="gsigma1"
!{(135,22)}*=0{\Gamma(\sigma_1)}
!{(100,35)}*=0{\bullet}="gsigma2"
!{(95,32)}*=0{\Gamma(\sigma_2)}
!{(120,30)}*=0{\bullet}="gtau0"
!{(112,30)}*=0{\Gamma(\tau_0)}
!{(140,35)}*=0{\bullet}="gtau1"
!{(145,32)}*=0{\Gamma(\tau_1)}
!{(90,45)}*=0{\bullet}="gtau2"
!{(90,49)}*=0{\Gamma(\tau_2)}
!{(110,40)}*=0{\bullet}="gnu0"
!{(115,44)}*=0{\Gamma(\nu_0)}
!{(130,45)}*=0{\bullet}="gnu1"
!{(125,49)}*=0{\Gamma(\nu_1)}
!{(100,50)}*=0{\bullet}="gmu0"
!{(102,54)}*=0{\Gamma(\mu_0)}
!{(140,55)}*=0{\bullet}="gmu1"
!{(142,59)}*=0{\Gamma(\mu_1)}
"root"-"rho"
"rho"-"rho0"
"rho"-"rho1"
"rho0"-"sigma"
"sigma"-"sigma0"
"sigma"-"sigma1"
"sigma0"-"sigma1.5"
"sigma1.5"-"sigma2"
"sigma1.5"-"mu"
"mu"-"mu0"
"mu"-"mu1"
"rho1"-"tau0"
"rho1"-"tau0.5"
"tau0.5"-"tau1"
"tau0.5"-"tau2"
"tau0"-"nu0"
"tau0"-"nu1"
"groot"-"grho"
"grho"-"gsigma0"
"grho"-"gsigma1"
"gsigma0"-"gsigma2"
"gsigma0"-"gtau0"
"gsigma1"-"gtau1"
"gsigma2"-"gtau2"
"gtau0"-"gnu0"
"gtau0"-"gnu1"
"gnu0"-"gmu0"
"gnu1"-"gmu1"
}\]